\documentclass{amsart}

\usepackage{amsmath, amssymb, amsthm}
\usepackage[all]{xy}

\newtheorem{thm}{Theorem}[section]

\newtheorem{cor}[thm]{Corollary}
\newtheorem{prop}[thm]{Proposition}
\newtheorem{lem}[thm]{Lemma}

\DeclareMathOperator{\soc}{soc}
\DeclareMathOperator{\Ext}{Ext}
\DeclareMathOperator{\Ann}{Ann}
\DeclareMathOperator{\tr}{tr}

\begin{document}

\title[Cartan matrices of $p$-blocks]{On diagonal entries of \\ Cartan matrices of $p$-blocks}
\author[Y. Otokita]{Yoshihiro Otokita}
\address{Yoshihiro Otokita: \newline
Department of Mathematics and Informatics, \newline
Graduate School of Science, \newline
Chiba University, \newline
1-33 Yayoi-cho, Inage-ku, Chiba-shi, 263--8522, \newline
Japan.}
\email{otokita@chiba-u.jp}

\maketitle

\begin{abstract}
In this short note, we show some inequalities on Cartan matrices, centers and socles of blocks of group algebras. Our main theorems are generalizations of the facts on dimensions of Reynolds ideals. 
\end{abstract}

\section{Introduction}
Let $p$ be a prime, $G$ a finite group and $(K, \mathcal{O}, F)$ a splitting $p$-modular system for $G$ where $\mathcal{O}$ is a complete discrete valuation ring with quotient field $K$ of characteristic $0$ and residue field $F$ of characteristic $p$. For each block $B$ of the group algebra $FG$, we denote by $k(B)$ and $l(B)$ the numbers of irreducible ordinary and Brauer characters associated to $B$, respectively. Our purpose of this note is to show some inequalities on the Cartan matrix $C_{B}$, the center $Z(B)$ and the $n$-th socle $\soc^{n} (B)$ of $B$. In this note, for any integer $n \ge 1$, $c(B, n)$ denotes the sum of multiplicities of $S$ as composition factors in the factor module $P_{S} / P_{S}J^{n}$ where $S$ ranges over isomorphism classes of irreducible right $B$-modules, $P_{S}$ is the projective cover of $S$ and $J$ is the Jacobson radical of $B$. Therefore, for example, $c(B, 1) = l(B), c(B, 2) = l(B) + \sum_{S} \dim \Ext^{1}_{B}(S, S)$ and $c(B, \lambda) = \tr C_{B}$ for the Loewy length $\lambda$ of $B$. In this note, we prove the following theorems. 

\begin{thm} \label{Thm1.1} For any $1 \le n \le \lambda$, 
\[ \dim \soc^{n}(B) \cap Z(B) \le c(B, n). \]
\end{thm}

\begin{thm} \label{Thm1.2} Assume $2 \le \lambda$. Then there exists an integer $2 \le m \le \lambda$ such that 
\begin{align*}
& \dim \soc^{n}(B) \cap Z(B) = c(B, n), \\
& \dim \soc^{n'}(B) \cap Z(B) \lneq c(B, n')
\end{align*}
for all $1 \le n \le m < n' \le \lambda$. 
\end{thm}

We mention some previous results as remarks of the main theorems. It is well known that $\dim \soc(B) \cap Z(B) = l(B)$ and Okuyama has shown that $\dim \soc^{2} (B) \cap Z(B) = l(B) + \sum_{S} \dim \Ext^{1}_{B} (S, S)$ in \cite{O} (the article is written in Japanese, so see \cite{Ko} for the original proof by Okuyama or see Proposition \ref{Prop3.3} in this paper). Therefore Theorem \ref{Thm1.2} is a generalization of these facts. Moreover we can obtain the famous inequality $k(B) \le \tr C_{B}$ as a corollary to Theorem \ref{Thm1.1} as will be proved later. 

\section{preliminaries}
This section is devoted to some notations and fundamental properties of finite-dimensional symmetric algebra $A$ over $F$ with a bilinear form $< \ , \ > \ : A \times A \to F$. The facts described in this section are applied to the basic algebra of $B$.   

For a subspace $U$ of $A$, we define 
\begin{align*}
& \Ann_{A}(U) = \{ a \in A \ | \ Ua = 0 \}, \\
& U^{\perp} = \{ a \in A \ | \ <U, a> = 0 \}.
\end{align*}

\begin{lem} \label{Lem2.1} Let $U, V$ be two subspaces of $A$. Then the following hold:
\begin{enumerate}
\item $(U^{\perp})^{\perp}=U$. \\
$(U + V)^{\perp} = U^{\perp} \cap V^{\perp}$. \\
$(U \cap V)^{\perp} = U^{\perp} + V^{\perp}$.
\item If $V \subseteq U$, then $U^{\perp} \subseteq V^{\perp}$.
\item $\dim U^{\perp} = \dim A - \dim U$.
\item If $U$ is an ideal of $A$, then $\Ann_{A}(U) = U^{\perp}$.
\end{enumerate}
\end{lem}

Furthermore, we define the \textit{commutator subspace} of subspaces $U$ and $V$ of $A$ by
\[ [U, V] = \sum_{u \in U, v \in V} F(uv - vu). \]
By the definition above, the next lemma is clear.

\begin{lem} \label{Lem2.2} Let $U, V$ and $W$ be subspaces of $A$. Then we have 
\begin{align*}
& [U + V, W] = [U, W] + [V, W], \\
& [U, V + W] = [U, V] + [U, W].
\end{align*}
\end{lem}

In particular, the next lemma is important.

\begin{lem} [{\cite[Lemma A]{K}}] \label{Lem2.3} $[A, A]^{\perp} = Z(A)$. \end{lem}

Now let $e_{1}, \dots, e_{l(B)}$ be representatives for the conjugacy classes of primitive idempotents in $B$. Thus $c_{ij} = \dim e_{i}Be_{j}$ and $c(B, n) = \sum_{1 \le i \le l(B)} \dim e_{i}Be_{i}/e_{i}J^{n}e_{i}$ where $C_{B} = (c_{ij})$. We put $e = e_{1} + \dots + e_{l(B)}$ and denote by $eBe$ the basic algebra of $B$. Then $B$ and $eBe$ are symmetric algebras over $F$. Moreover they are Morita equivalent since $B=BeB$, and hence the next lemma holds. 

\begin{lem} \label{Lem2.4} For an ideal $I$ of $B$, $eIe$ is that of $eBe$ and we have
\[ \dim \Ann_{B}(I) \cap Z(B) = \dim \Ann_{eBe} (eIe) \cap Z(eBe). \]
\end{lem}

Finally, we define a subspace 
\[ B(n) = \sum_{1 \le i \le l(B)} e_{i} J ^{n} e_{i} + \sum_{1 \le i \neq j \le l(B)} e_{i}Be_{j} \]
of $eBe$ for each $n \ge 1$. Since $eBe = \sum_{1 \le i, j \le l(B)} e_{i}Be_{j}$ and $B(n)$ are direct sums, we deduce the next lemma by Lemma \ref{Lem2.1}.

\begin{lem} \label{Lem2.5} 
\begin{align*}
& \dim eBe = \sum_{1 \le i, j \le l(B)} c_{ij}, \\
& \dim B(n)^{\perp} = c(B, n).
\end{align*}
\end{lem}

\section{Proof of main theorems}
Theorem \ref{Thm1.1} is due to the next lemma.

\begin{lem} \label{Lem3.1} If $i \neq j$, then $e_{i}Be_{j} \subseteq [eBe, eBe]$. \end{lem}
\begin{proof}
For any $x \in e_{i}Be_{j}$, we can write $x = xe_{j} - e_{j}x$. So $x \in [e_{i}Be_{j}, e_{j}Be_{j}] \subseteq [eBe, eBe]$. 
\end{proof}

Now we prove Theorem \ref{Thm1.1}. 

\begin{proof}[\textit{Proof of Theorem \ref{Thm1.1}}]
By Lemma \ref{Lem3.1}, $B(n) \subseteq eJ^{n}e + [eBe, eBe]$ and hence $\Ann_{eBe}(eJ^{n}e) \cap Z(eBe) \subseteq B(n)^{\perp}$ using Lemma \ref{Lem2.1} and \ref{Lem2.3}. So Lemma \ref{Lem2.4} and \ref{Lem2.5} gives us that $\dim \soc^{n}(B) \cap Z(B) = \dim \Ann_{B}(J^{n}) \cap Z(B) \le c(B, n)$ as claimed.
\end{proof}

We show a corollary to Theorem \ref{Thm1.1}. We substitute $\lambda$ for $n$ and thus we obtain the next inequality (remark $k(B)=\dim Z(B)$ and see also {\cite[Proposition 4.2]{B}} or {\cite[Theorem A]{W}}). 

\begin{cor} $k(B) \le \tr C_{B}$. \end{cor}

To prove Theorem \ref{Thm1.2}, we have to see the structure of $[eBe, eBe]$ as follows.

\begin{lem} \label{Lem3.2} $[eBe, eBe] \subseteq \sum_{1 \le i, j \le l(B)} [e_{i}Be_{j}, e_{j}Be_{i}] + \sum_{1 \le i \neq j \le l(B)} e_{i}Be_{j}$. \end{lem}
 
\begin{proof}
We first obtain $[eBe, eBe] = \sum_{1 \le i, j, s, t \le l(B)}[e_{i}Be_{j}, e_{s}Be_{t}]$ by Lemma \ref{Lem2.2}. Thereby we investigate three cases in the following. \\

\textbf{Case 1} $i \neq t$ and $j \neq s$. \\
Since $e_{t}e_{i} = e_{j}e_{s} = 0$, $[e_{i}Be_{j}, e_{s}Be_{t}] = 0$. \\

\textbf{Case 2} $i=t$ and $j \neq s$.\\
Clearly, we have $[e_{i}Be_{j}, e_{s}Be_{i}] \subseteq e_{s}Be_{i}Be_{j} \subseteq e_{s}Be_{j}$ where $j \neq s$. \\

\textbf{Case 3} $i \neq t$ and $j = s$. \\
By similar way above, $[e_{i}Be_{j}, e_{j}Be_{t}] \subseteq e_{i}Be_{t}$ where $i \neq t$.\\

Thus the claim follows since the remaining case is that $i = t$ and $j = s$.
\end{proof}

Theorem \ref{Thm1.2} is a direct consequence of the next proposition. 
\begin{prop} \label{Prop3.3} The following are equivalent:
\begin{enumerate}
\item $\dim \soc^{n}(B) \cap Z(B) = c(B, n)$.
\item $[e_{i}Be_{j}, e_{j}Be_{i}] \subseteq e_{i}J^{n}e_{i} + e_{j}J^{n}e_{j}$ for all $1 \le i, j \le l(B)$.
\end{enumerate}
\end{prop}
\begin{proof}
By the proof of Theorem \ref{Thm1.1}, (1) holds if and only if $eJ^{n}e + [eBe, eBe] \subseteq B(n)$. However, by Lemma \ref{Lem3.1} and \ref{Lem3.2}, we have
\begin{align*}
eJ^{n}e + [eBe, eBe] & = \sum_{1 \le i \le l(B)} e_{i}J^{n}e_{i} + [eBe, eBe] \\
                                  & =B(n) + \sum_{1 \le i, j \le l(B)}[e_{i}Be_{j}, e_{j}Be_{i}].
\end{align*}

So it is clear that (2) implies (1). On the other hand, suppose (1) holds and $x \in [e_{i}Be_{j}, e_{j}Be_{i}]$. Thus $x \in [eBe, eBe] \subseteq B(n)$. Since we can write $x=e_{i}xe_{i} + e_{j}xe_{j}$, we deduce that $x \in e_{i}J^{n}e_{i} + e_{j}J^{n}e_{j}$, as required.
\end{proof}
  
It is easy to show the result of Okuyama in \cite{O} by using this proposition. If $i \neq j$, then $e_{i}Be_{j} = e_{i}Je_{j}$ and so $[e_{i}Be_{j}, e_{j}Be_{i}] = [e_{i}Je_{j}, e_{j}Je_{i}] \subseteq e_{i}J^{2}e_{i} + e_{j}J^{2}e_{j}$. On the other hand, $[e_{i}Be_{i}, e_{i}Be_{i}] = [Fe_{i} + e_{i}Je_{i}, Fe_{i} + e_{i}Je_{i}] \subseteq e_{i}J^{2}e_{i}$ since $e_{i}Be_{i}$ is local. Therefore $n = 2$ satisfies the conditions above. 

Now we prove Theorem \ref{Thm1.2}. We fix the largest integer $2 \le m \le \lambda$ satisfies the conditions in Proposition \ref{Prop3.3}. Then any integer $n$ (such that $n \le m$) holds same properties since $J^{m} \subseteq J^{n}$. Therefore we have completed the proof of the theorem.

\section*{Acknowledgment}
The author would like to thank Shigeo Koshitani and Taro Sakurai for helpful discussions and comments.

\end{document}